\documentclass[english]{amsart}

\usepackage{amsfonts,amsmath,amssymb,color,txfonts,amsthm,graphics,mathrsfs}

\usepackage[T1]{fontenc}
\usepackage[english]{babel}
\usepackage{color}
\usepackage[T1]{fontenc}
\usepackage[all]{xy}

\usepackage{enumerate}
\usepackage{amscd}
\usepackage{exscale}
\usepackage{latexsym}
\usepackage[all]{xy}
\usepackage{hyperref}
\usepackage{fancyhdr}
\usepackage{layout}
\usepackage{graphicx}

\newtheorem{theorem}{Theorem}
\newtheorem{lemma}{Lemma}[section]

\newtheorem{proposition}[lemma]{Proposition}

\theoremstyle{definition}

\newcommand\NN{{\mathbb N}}

\newcommand\QQ{{\mathbb Q}}

\newcommand\ZZ{{\mathbb Z}}

\newcommand\p{{\mathbb P}}

\newcommand\tr{\hbox to 1mm  {${}^t \!  $} }

\newcommand{\nc}{\newcommand}
\nc{\apl}{K\cup \{\infty\}}
\nc{\mfp}{\mathfrak{p}}
\nc{\dmfp}{\delta_\mfp}
\nc{\vmfp}{v_\mfp}
\nc{\ace}{\`e }
\nc{\aca}{\`a }
\nc{\aci}{\`i }
\nc{\aco}{\`o }
\nc{\acu}{\`u }
\nc{\pid}{\mathfrak{p} }
\nc{\bdm}{\begin{displaymath}}
\nc{\edm}{\end{displaymath}}
\nc{\beq}{\begin{equation}}
\nc{\eeq}{\end{equation}}
\nc{\dpid}{\delta_{\mathfrak{p}}}
\nc{\vvs}{\textquotedblleft}
\nc{\vvd}{\textquotedblright}
\nc{\os}{\mathcal{O_S}}

\nc{\mcu}{\mathcal{U}}

\title[Preperiodic points for rational functions over function fields]{Preperiodic points for rational functions defined over a rational function field of characteristic zero}

\author{Jung Kyu Canci}
\address{Jung Kyu Canci, Universit\"{a}t Basel, Mathematisches Institut, Rheinsprung $21$, CH-$4051$ Basel}
\email{jungkyu.canci@unibas.ch}
\begin{document}

\begin{abstract}
Let $k$ be an algebraic closed field of characteristic zero. Let $K$ be the rational function field $K=k(t)$. Let $\phi$ be a non isotrivial rational function in $K(z)$. We prove a bound for the cardinality of the set of $K$--rational preperiodic points for $\phi$ in terms of the number of places of bad reduction and the degree $d$ of $\phi$. \end{abstract}

\maketitle

\section{Introduction}Let $k$ be an algebraically closed field of characteristic zero. Let $K$ be the rational function field $K=k(t)$. Let $\phi\colon \p_1\to\p_1$ be an algebraic endomorphism of the projective line defined over $K$. 
For each non negative integer $m$ we denote with $\phi^m$ the $m$--iterate of $\phi$, where $\phi^0$ denotes the identity map. 
We say that $P\in \p_1(K)$ is \emph{periodic} for $\phi$ with \emph{minimal} period $n$ if $\phi^n(P)=P$, $n>0$ and  $\phi^e(P)\neq P$ for each $0< e<n$. We say that $P$ is \emph{preperiodic} for $\phi$ if there exist $n,m\in \NN$ with $n>0$ such that $\phi^{n+m}(P)=\phi^m(P)$. In other words, $P$ is preperiodic if its orbit with respect $\phi$, i.e.
$$O_\phi(P)=\{\phi^n(P)\mid n\in\NN\},$$
is finite. 
We denote by ${\rm PrePer}(\phi,K)$ the set of prepepriodic points in $\p_1(K)$ for $\phi$. 
For each endomorphism  $\phi$ and each element  $A\in {\rm PGL}_2(K)$ (the group of automorphisms of $\p_1$ defined over $K$) we denote by $\phi^A$ the conjugate $A\circ \phi\circ A^{-1}$ of $\phi$ by $A$. We consider this action of ${\rm PGL}_2(K)$ because it does not change the dynamics, in the sense that $\phi$ and $\phi^A$ determine the same dynamic for each $A\in {\rm PGL}_2(K)$.

Let us denote by $\overline{K}$ the algebraic closure of $K$. An endomorphism $\phi$ of $\p_1$, defined over $K$, is said \emph{isotrivial} if there exists an automorphism $A\in {\rm PGL}_2(\overline{K})$ such that the map $\phi^A$ is defined over $k$. We say that $\phi$ is \emph{isotrivial over $K$} if there exists an automorphism $A\in {\rm PGL}_2({K})$ such that the map $\phi^A$ is defined over $k$. The fact that $k$ is algebraically closed implies that the set  ${\rm PrePer}(\phi,K)$ is an infinite set if $\phi$ is defined over $k$ (see for example \cite{IB}). Therefore in order to obtain a finiteness result for the set ${\rm PrePer}(\phi,K)$, $\phi$ should be non isotrivial (over $K$).

In the present work we shall take in consideration the notion of good reduction for endomorphisms of $\p_1$ introduced by Morton and Silverman in \cite{MS}. In literature, it is sometimes called \emph{simple} good reduction in order to make difference from other notions of good reduction. In the present work we will take in consideration a slightly modified definition of good reduction. We say that $\phi$ has \emph{good reduction} at a place $\pid$, if there exists an automorphism $A\in {\rm PGL}_2(K)$ such that the map $A\circ \phi\circ A^{-1}$ has simple good reduction at $\pid$. We will recall in the next section the definition of simple good reduction. According to our definitions there are maps that have good reduction but not simple good reduction at some places. E.g. consider the map $\phi([X:Y])=[tX^2:Y^2]$ that has bad simple reduction at the place given by $t$, but it has good reduction at $t$ by considering the automorphism $A([X:Y])=[tX:Y]$, since $A\circ\phi\circ A^{-1}([X:Y])=[X^2:Y^2])$.

Recall that there is a natural identification between the elements of the set of endomorphisms of $\p_1$ and the elements in $K(z)$, i.e. rational functions defined over $K$, and a corresponding identification of $\p_1(K)$ with $K\cup \{\infty\}$, where each point $[x:y]\in \p_1(K)$ is identified with the point $x/y$ in $K\cup \{\infty\}$. 

     To ease notation we shall consider rational functions instead of endomorphisms of $\p_1$.


Our main result is the following one:

\begin{theorem}\label{mainT}
Let $k$ be an algebraically closed field of characteristic zero. Let $K$ be the rational function field $K=k(t)$. Let $S$ be a finite set of places of $K$ of cardinality $s\geq 1$. Let $d$ be a positive integer. Then there exists a bound $B(d,s)$ such that for each  $\phi\in K(z)$ non isotrivial over $K$ of degree $d$, with good reduction outside $S$, the inequality
$$\#{\rm PrePer}(\phi,K)\leq B(d,s)$$
holds.
\end{theorem}

The outline of the proof of Theorem \ref{mainT} is similar to the one for \cite[Corollary 1.1]{CP}. But in our setting we have the problem that the residue fields are infinite, more precisely they are the algebraically closed field $k$. With global fields one has the advantage that each residue field is finite. This fact plays an important role in almost all proofs in  \cite{CP}. In our situation we need some preliminary lemmas, in particular Lemma \ref{pairs} and Lemma \ref{cp}, to bypass the problem about the infiniteness of the reduced field. 

The condition $S$ not empty is only a technical one and it does not represent a significant restriction.

Next result is an important tool in the proof of Theorem \ref{mainT}.
\begin{theorem}\label{Tper}
Let $k$ be an algebraically closed field of characteristic zero. Let $K$ be the rational function field $K=k(t)$.  Let $S$ be a finite set of places of $K$  of cardinality $s\geq 1$. Let $\phi\in K(z)$ be non isotrivial over $K$ of degree $d$ with good reduction outside $S$. Let $P$ be defined over $K$ and be a periodic point for $\phi$, with minimal period $n$. Then
\beq \label{Bc}n\leq \prod_{p\leq b(d,s)\atop \text{$p$ prime}} \max\left\{\frac{9^{s-1}+1}{2}(2d+1)+2,  p\cdot 3^{2s-1}\right\},\eeq
where $b(d,s)= \frac{9^{s-1}+1}{2}(2d+1)+2$.\end{theorem}
\noindent To ease notation we will denote by $C(d,s)$ the bound in (\ref{Bc}).

Some ideas to prove Theorem \ref{Tper} come from the proof of \cite[Theorem 7]{CP}. But in  \cite[Theorem 7]{CP}, it plays an important role the description of the possible shape of the minimal periodicity for a periodic points as given in \cite{MS1} or \cite{Zie}. These results do not apply to our situation, because the residue fields are infinite. Therefore in our proof of Theorem \ref{Tper} we use new ideas, mainly contained in Section \ref{prel}, in order to compensate the absence of results as the ones in \cite{MS1} and \cite{Zie}.

The second important tool in the proof of Theorem \ref{mainT} is the following result.
\begin{theorem}\label{Tpreper}
Let $k$ be an algebraically closed field of characteristic zero. Let $K$ be the rational function field $K=k(t)$. Let $S$ be a finite set of places of $K$  of cardinality $s\geq 1$. Let $\phi\in K(z)$ be non isotrivial over $K$ of degree $d$ with good reduction outside $S$. Let $P$  be defined over $K$ and be a preperiodic point for $\phi$. Then
\beq\label{Bpp}\#\left(O_\phi(P)\right)\leq \frac{9^{s+1}+1}{2}(2d+2)\prod_{p\leq b(d,s)\atop \text{$p$ prime}}\max\left\{\frac{9^{s-1}+1}{2}(2d+1)+2,  p\cdot 3^{2s-1}\right\},\eeq
where $b(d,s)= \frac{9^{s-1}+1}{2}(2d+1)+2$.\end{theorem}
\noindent To ease notation we will denote by $D(d,s)$ the bound in (\ref{Bpp}).


The main tool to prove the above theorems is a result about the finiteness of solutions in $S$--units of linear equations, namely \cite[Corollary 4]{Z1}. The idea of using $S$--units in the arithmetic of dynamical systems is due to Narkiewicz (see \cite{N.1}), where he studied dynamics associated to polynomials instead of generic rational functions. 

Some similar results obtained by using $S$--units equations are contained in \cite{C.1},\cite{C}. But all those results are concerning rational functions defined over number fields, where the linear equations have finite set of non degenerate solutions in $S$--units.
In our setting, a linear equation, even with only two addends, can have an infinite set of solutions in $S$--units. One has this problem also when the field $k$ is a finite field. But the infiniteness of the solutions in the case of global function fields is more manageable than the one in our situation. Indeed, for example in our setting we need to ask that the map $\phi$ is non isotrivial (condition that is not necessary in the case of global function field). This problem, about $S$--unit solutions of linear equations, represents another reason for the need of new ideas in addition to the ones already used in \cite{C.1}, \cite{C} and \cite{CP}.

Some results similar to our are given by Morton and Silverman in \cite{MS1} and by Benedetto in \cite{B} and \cite{B.1}. More precisely in \cite{B} there is a result (Theorem A) that characterizes preperiodic points,  for polynomials defined over an arbitrary algebraic function field of dimension one, in terms of the 
canonical height associated to $\phi$. Benedetto in \cite[Remark 5.2]{B} (and in the introduction of \cite{B}) affirms that Theorem A implies, for the non isotrivial polynomials, a bound for the cardinality of the set of preperiodic points, which has a size of the type $O(s\log s)$ (where the O--big constant depends on the degree $d$ of the polynomial). Also the bounds in 
\cite{MS1} and \cite{B.1} are of the form $O(s\log s)$. 

Our techniques are completely different from the ones applied in \cite{B}, \cite{B.1}  and \cite{MS1}. Our methods lead on to prove  some bounds that are worse than the ones in \cite{B}, \cite{B.1}  and \cite{MS1}. But our results hold for rational functions and for preperiodic points.

It is possible to give an explicit value for the bound $B(d,s)$ in Theorem \ref{mainT},  which is presented in the proof of Theorem \ref{mainT}. We have omitted to write it in the statement of Theorem \ref{mainT} because it is huge. The aim of the present work was just to prove the existence of a bound $B(d,s)$ as described in Theorem \ref{mainT}.

Our work is linked with Morton and Silverman's conjecture also known as \emph{Uniform Boundedness Conjecture} for dynamical systems (see \cite{MS1}).  It affirms that for any number field $K$, the cardinality of the set ${\rm PrePer}(\phi,K)$ of a morphism $\phi\colon \p_N\to\p_N$ of degree $d \geq 2$, defined over $K$, is bounded by a number depending only on the integers $d,N$ and on the degree $D$ of the extension $K/\QQ$.
Our Theorem \ref{mainT} gives a statement in the direction of the Uniform Boundedness Conjecture in the case when $K$ is an algebraic function field. 

Note that in the case when $K$ is a number field, in order to have finiteness of the set ${\rm PrePer}(\phi,K)$, it is necessary to set the condition that the degree $d$ of $\phi$ is $\geq 2$. Indeed for each automorphism $A\in {\rm PGL}_2(K)$ of finite order, i.e. $A^n=Id$ for a positive $n$, the set  ${\rm PrePer}(A,K)$ is the whole $\p_1(K)$. In our setting, a non isotrivial $\phi$, even of degree 1, can not have finite order. 

The results proven in this article should have an analogue to any finite extension of $k(t)$ (i.e. algebraic function fields over $k$). The fact that the ring of $S$--integers in $K=k(t)$ is a principal ideal domain plays a crucial role in our proofs. This fact is not true if $K$ is the function field of a curve of positive genus. It would be interesting to find a generalization of our result to any algebraic function fields.

\emph{Acknowledgements.} The author would like to thank Laura De Marco for setting the question, which has motivated the present work. He would like to thank also for the helpful discussions with Laura Paladino, Laura De Marco and Pietro Corvaja. The author would like to thank also Joe Silverman for pointing out a mistake in a previous version of the present work. Furthermore he would like to thank Sebastian Troncoso for suggesting an improvement of the bound in Lemma \ref{pairs}.

%
%
%
\section{Notation and definitions}\label{notdef}
Throughout all the paper we shall use the following notation: $k$ will be an algebraic closed field of characteristic zero; $K$ will be the rational function field $k(t)$. We denote by $R$ the  polynomial ring $k[t]$. Any place of $K$ is either given by a valuation $v_\alpha$ for $\alpha\in k$ or given by the valuation $v_\infty$ at infinity (associated to the point at infinity in $\p_1$). A valuation $v_\alpha$ is the one associated to the polynomial $t-\alpha$.  The valuation at infinity $v_\infty$ is the one associated to the element $1/t$. For any place $\mfp$ of $K$, we will denote the associated valuation by $v_\mfp$, normalized so that $v_\mfp(K)=\ZZ$. See \cite{Ros} or \cite{Sti} for the properties of places over function fields. 

For a finite set $S$ of places of $K$ of cardinality $s$, we set
$$R_S\coloneqq\{x\in K\mid \vmfp(x)\geq 0\ \  \text{for all $\mfp\notin S$}\}$$ 
the ring of $S$--integers and 
$$R_S^*\coloneqq\{x\in K\mid \vmfp(x)=0\ \  \text{for all $\mfp\notin S$}\}$$ 
the group of $S$--units. 

We will always assume that $S$ is an arbitrary fixed non empty finite set of places of $K$. 
We will denote by $s$ the cardinality of $S$. The rank of $R_S^*/k$ is $s-1$. 
The ring $R_S$ is a principal ideal domain and $K$ is the fraction field of the ring $R_S$. Therefore, each point of $K$ can be written in \emph{$S$--coprime integral form}; which means that for each $x\in K$ we may assume that $x=a/b$ with $a,b\in R_S$ and $\min\{v_\mfp(a),v_\mfp(b)\}=0$ for each $\mfp\notin S$ (in this case $a$ and $b$ are said \emph{$S$--coprime}).  

For each given $\phi\in K(z)$ of degree $d$, there exist $f(z),g(z)\in R_S[z]$, coprime polynomials, such that 
\beq\label{nform}\phi(z)=f(z)/g(z), \text{ with } f(z)=f_dz^d+\ldots +f_1z+f_0,\ \ g(z)=g_dz^d+\ldots +g_1z+g_0\eeq
for suitable $f_d,\ldots,g_0\in R_S$ with no common factors in $R_S\setminus k$. When $\phi$ is written in the above form, we shall say that it is written in a \emph{$S$--reduced integral form}. 

Let $\phi$ be written in a $S$--reduced integral form as in (\ref{nform}). We say that $\phi$ has \emph{simple good reduction outside $S$} if the homogeneous resultant of $f$ and $g$ is in $R_S^*$. Here the homogeneous resultant is the one with respect the homogenized polynomials of degree $d$ obtained from $f$ and $g$. 
 
 For each place $\mfp$, recall that the residue field $\mathcal{O}_\mfp/\mathcal{M}_\mfp$, obtained as the quotient of the ring $\mathcal{O}_\mfp=\{x\in K\mid v_\mfp(x)\geq 0\}$ by its maximal ideal $\mathcal{M}_\mfp=\{x\in K\mid v_\mfp(x)> 0\}$, is $k$. For each element $x\in \mathcal{O}_\mfp $, we will denote by $\overline{x}$ the corresponding image of $x$ in $\mathcal{O}_\mfp/\mathcal{M}_\mfp$ of the canonical projection $\mathcal{O}_\mfp\to \mathcal{O}_\mfp/\mathcal{M}_\mfp$. The element $\overline{x}$  is called \emph{the reduction modulo $\mfp$ of $x$}. If $x\notin \mathcal{O}_\mfp$, its reduction modulo $p$ is the point at infinity $\infty=1/0$.
 
If a map $\phi$ is written in $S$--reduced integral form, then it is well defined the map $\phi_\mfp\in k(z)$ obtained from $\phi$ by reduction of its coefficients modulo $\mfp$, for each $\mfp\notin S$. The condition to be of simple good reduction can be reformulated in the following way: $\phi$ has simple good reduction outside $S$ if and only if $\deg \phi_\mfp=\deg\phi$ for all $\mfp\notin S$. 

As an application of  \cite[Corollary 2.13]{BM} by Bruin and Molnar, one sees that each rational map $\phi$ of degree $d$ defined over $K$ admits a $R$--minimal model. By using the notation in \cite{BM} we have that there exists $\psi=A\circ\phi\circ A^{-1}$, for a suitable $A\in{\rm PGL}_2(K)$, with $\psi=F(z)/G(z)$ written in reduced form for each place $\pid$ such that 
$${\rm Res}_R([\phi])={\rm Res}_d(F,G)R.$$

Therefore a rational function $\phi(t)$ defined over $K$ has \emph{good reduction} at $\pid$ if a $R$--minimal model of $\phi$ has simple good reduction at $\pid$. 
For a generalization of \cite[Corollary 2.13]{BM}, see \cite{PS}.

Let $P_1=x_1/y_1,P_2=x_2/y_2$ be two distinct points in $\apl$. By using the notation of  \cite{MS} we shall denote by
\begin{center}$\delta_{\mathfrak{p}}\,(P_1,P_2)=v_{\mathfrak{p}}\,(x_1y_2-x_2y_1)-\min\{v_{\mathfrak{p}}(x_1),v_{\mathfrak{p}}(y_1)\}-\min\{v_{\mathfrak{p}}(x_2),v_{\mathfrak{p}}(y_2)\}$\end{center}
the \emph{$\mathfrak{p}$-adic logarithmic distance}. The logarithmic distance is always non negative and $\delta_{\mathfrak{p}}(P_1,P_2)>0$ if and only if $P_1$ and $P_2$ have the same reduction modulo $\pid$. Note that if $P_1$ and $P_2$ are written in $S$--coprime integral form, then $\delta_{\mathfrak{p}}\,(P_1,P_2)=v_{\mathfrak{p}}\,(x_1y_2-x_2y_1)$ $\forall \pid\notin S$. 

Let ${\rm PGL}_2(R_S)$ be the group of automorphisms defined by a matrix in $GL_2(R_S)$ whose discriminant is in $R_S^*$, i.e. a invertible matrix with entries in $R_S$, whose inverse has entries in $R_S$. Sometimes we will take  in consideration the conjugate of a map $\phi$ with an automorphism $A\in{\rm PGL}_2(R_S)$.  We will use the fact that 
$\delta_{\mathfrak{p}}\,(A(P_1),A(P_2))=\delta_{\mathfrak{p}}\,(P_1,P_2)$
for each $A\in  {\rm PGL}_2(R_S)$. 

\section{Auxiliary results}
An important tool in our proof will be the following result by Zannier in \cite[Corollary 4]{Z1} that here we present in a form adapted to our setting.

\begin{theorem}\label{csu}\emph{\cite[Corollary 4]{Z1}}
Let $\lambda,\mu\in K^*$. Then the equation $\lambda x+\mu y=1$ has at most $9^{s-1}$ solutions $(x,y)\in (R_S^*)^2$ such that $\lambda x/\mu y\notin k^*$.
\end{theorem}

Zannier's article concerns problems as the one in Theorem \ref{csu} but in a much more general setting than our. For some more recent and general results see also \cite{EZ}. 

The divisibility arguments, that we shall use to produce the $S$--unit equations useful to prove our bounds, are obtained starting from the following two facts:

\begin{proposition}\label{5.1}\emph{\cite[Proposition 5.1]{MS}} For all $P_1,P_2,P_3\in\apl$, we have
\par\medskip \centerline{$\delta_{\mathfrak{p}}(P_1,P_3)\geq \min\{\delta_{\mathfrak{p}}(P_1,P_2),\delta_{\mathfrak{p}}(P_2,P_3)\}$.}\end{proposition}

\begin{proposition}\label{5.2}\emph{\cite[Proposition 5.2]{MS}}
Let $\phi\in K(t)$ be a rational function with simple good reduction outside $S$. Then for any $P,Q\in\apl$ we have
$\delta_{\mathfrak{p}}(\phi(P),\phi(Q))\geq \delta_{\mathfrak{p}}(P,Q)$ for each $\mfp\notin S$.
\end{proposition}

As a direct application of the previous propositions we have the next result.

\begin{proposition}\label{6.1}\emph{\cite[Proposition 6.1]{MS}}
Let $\phi\in K(t)$ be a rational function with simple good reduction outside $S$.  Let $P\in\apl$ be a periodic point for $\phi$ with
minimal period n. Then the following hold for each  $\mfp\notin S$.
\begin{itemize}
\item $\dpid(\phi^i(P),\phi^j(P))=\dpid(\phi^{i+k}(P),\phi^{j+k}(P))$\ \ for every $i,j,k\in\NN$,
\item Let $i,j\in\NN$ be  such that $\gcd(i-j,n)=1$. Then $\dpid(\phi^i(P),\phi^j(P))=\dpid(\phi(P),P)$.
\end{itemize}
\end{proposition}
\section{Proofs}
The section is divided in several subsections containing the proofs of some lemmas, the proof of Theorem \ref{Tper}, the proof of Theorem \ref{Tpreper} and finally the proof of Theorem \ref{mainT}.

\subsection{Preliminary lemmas} \label{prel}
We start by giving a very simple lemma, whose proof is an elementary application of Theorem \ref{csu}. Recall that $K$ and $S$ are defined as described in Section \ref{notdef}. 

\begin{lemma}\label{goodeq}
Let $\lambda,\mu\in K^*$ be such that there exists a place $\mfp\notin S$ with $v_\mfp(\lambda)\neq v_\mfp(\mu)$. Then the equation $\lambda x+\mu y=1$ has at most $9^{s-1}$ solutions $(x,y)\in (R_S^*)^2$.
\end{lemma}
\begin{proof}
The proof is an application of Theorem \ref{csu}, because there are no $(x,y)\in (R_S^*)^2$ such that $\lambda x/\mu y\in k$. 
\end{proof}

Next lemma contains in the hypothesis the crucial condition of non triviality for the maps.

\begin{lemma}\label{pairs}
Let $\phi\in K(z)$ be a rational function defined of degree $d$ not isotrivial over $K$. The set of pairs $(\lambda_1,\lambda_2)\in (k)^2$ such that $\phi(\lambda_1)=\lambda_2$ is finite and is bounded by $2d$.
\end{lemma}
\begin{proof} Since $\phi$ is non isotrivial, $d\geq 1$. 

We assume $\phi$ written in reduced normal form as in (\ref{nform})
for suitable $f_d,\ldots,g_0\in  R$ without common factors in $R\setminus k$.

The condition $\phi(\lambda_1)=\lambda_2$ is equivalent to say that the polynomial
$$
T(X,Y)\coloneqq \left(f_dX^d+\ldots +f_1X+f_0\right)-\left(g_dX^d+\ldots +g_1X+g_0\right)Y
$$
is zero at $(X,Y)=(\lambda_1,\lambda_2)$. The equivalence follows from the fact that the two plynomials $f(z),g(z)$ are coprime, so they do not have common roots. The polynomial $T(X,Y)$ is in $K\setminus k$, because $\phi$ is not isotrivial. Furthermore, because of our assumption on the coefficients $f_d,\ldots,g_0$, we have that $T(X,Y)$ is not factorisable in the form 
$$T(X,Y)=\alpha Q(X,Y)$$ 
with $\alpha\in K\setminus k$ and $Q(X,Y)$ polynomial defined over $k$. Moreover, we can see that the polynomial $T(X,Y)$ is irreducible in $K[X,Y]$. Suppose the contrary; since the degree of $T(X,Y)$ with respect $Y$ is one, then the decomposition should be of the form $T(X,Y)=P(X)Q(X,Y)$ for a suitable $P(X)\in K[X]\setminus K$ and $Q(X,Y)\in K[X,Y]\setminus K$. But this is an absurd. Indeed, take $\mfp$ a place of simple good reduction for $\phi$ such that the reduction $P_\mfp(X)$ has positive degree, where $P_\mfp(X)$ denotes the polynomial obtained by reduction modulo $\mfp$ of the coefficients of $P(X)$. Let $\alpha\in k$ be a root of $P_{\mfp}(X)$. The factorization $T(X,Y)=P(X)Q(X,Y)$ would imply that the map $\phi_{\mfp}$ send $\alpha$ to any $\gamma\in k$, that is clearly an absurd. In the last part we used that $\phi_\mfp(\overline{x})=\overline{\phi(x)}$ for each $x\in K$, which follows from the fact that $\phi$ has good reduction at $\mfp$ (see for example \cite[Theorem 2.18]{Sil.2}).


We may consider $T(X,Y)$ as a polynomial in $k[t,X,Y]$, that is possible by our choice of $f(z)$ and $g(z)$. Let us collect $T(X,Y)$ with respect $t$. If the degree of $T(X,Y)$ with respect $t$ is $n$, for each $i\in\{0,1,\ldots,n\}$, we denote by $h_i(X,Y)$ the polynomial in $k[X,Y]$ such that
$$T[X,Y]=h_n(X,Y)t^n+\ldots+h_1(X,Y)t+h_0(X,Y).$$
From the previous remarks we know that the polynomials $h_i$'s have no common irreducible factors. This is important to know because we have that $\phi(\lambda_1)=\lambda_2$ if and only if $h_i(\lambda_1,\lambda_2)=0$ for all $i\in\{0,1,\ldots,n\}$. Note that for each $i\in\{0,1,\ldots,n\}$, such that $h_i(X,Y)$ is not the zero polynomial, there exist two polynomials $a_i(X)\in k[X]$ and $b_i(X,Y)\in k[X,Y]$ such that $h_i(X,Y)=a_i(X)b_i(X,Y)$ and either $b_i(X,Y)=1$ or the degree of $b_i(X,Y)$ with respect $Y$ is exactly one and its total degree is bounded by $d+1$. In both cases the degree of the $a_i$'s is bounded by $d$. 
If there exists an index $i$ such that $b_i(X,Y)=1$ and $a_i(X)\neq 0$, then the number of solutions $(\lambda_1,\lambda_2)\in k^2$ of $T(X,Y)=0$ is bounded by $d$. Otherwise there are at most $2d$ solutions $(\lambda_1,\lambda_2)\in k^2$. 
%
\end{proof}

\begin{lemma}\label{cp}
Let $\phi\in K(z)$ be a rational function not isotrivial over $K$.   Let $d$ be the degree of $\phi$. Let $\{P_0,P_1,\ldots,P_{n-1}\}$ be a set of $n$ distinct points of $K\cup\{\infty\}$ with the property that $\phi(P_i)=P_{i+1}$, for each $i\in\{0,\ldots n-1\}$, and $\dpid(P_i,P_j)=\dpid(P_0,P_1)$ for each distinct $i,j\in \{0,\ldots n-2\}$ and for each $\mfp\notin S$. Then
\beq\label{n}n\leq \frac{9^{s-1}+1}{2}(2d+1)+2.\eeq
\end{lemma}

\begin{proof}We assume $n>2$ (otherwise the statement is trivially true). 
Up to conjugation of $\phi$ we may assume that $P_0=0$ and $P_1=\infty$. It is sufficient to take $A\in {\rm PGL}_2(K)$ that send $P_0$ to $0$ and $P_1$ to $\infty$ and take $A\circ\phi\circ A^{-1}$ (that is again non isotrivial) instead of $\phi$ and $A(P_i)$ instead of $P_i$. Note that $\dpid(A(P_i),A(P_j))=\dpid(P_i,P_j)-\dpid(P_0,P_1)=0$ for all distinct $i,j\in\{0,\ldots n-1\}$ and $\pid\notin S$. For each $i\in\{2,\ldots n-1\}$ let $x_i,y_i\in R_S$ such that $P_i=x_i/y_i$ is written in $S$--coprime integral form. By considering the fact that the $p$--adic distances $\dpid(P_0,P_i)=v_\mfp(x_i)$  and $\dpid(P_1,P_i)=v_\mfp(y_i)$ are zero for all $\pid\notin S$, one sees that there exists an $S$--unit $u_i$ such that $P_i=u_i$.  Furthermore we consider for each $i\in\{3,\ldots n-1\}$ the $\pid$--adic distance $\dpid(P_2,P_i)$ and we obtain that $u_i-u_2$ is a $S$--unit. Therefore there exists an $S$--unit $u_{2,i}$ such that the following equality hold:
$$\frac{u_i}{u_2}+\frac{u_{2,i}}{u_2}=1.$$ Hence $u_i$ is of the shape $u_i=u_2\lambda_i$, where $\lambda_i\in R_S^*$ is such that there exists $\mu_i\in R_S^*$ so that $\lambda_i+\mu_i=1$. Hence, we have the following situation:
$$0\stackrel{\phi}{\mapsto}\infty\stackrel{\phi}{\mapsto}u_2\stackrel{\phi}{\mapsto}u_2\lambda_3\stackrel{\phi}{\mapsto}\ldots \stackrel{\phi}{\mapsto}u_2\lambda_{n-1}.$$
Up to conjugation by $z\mapsto u_2^{-1}z$ we may assume that $u_2=1$. So we reduce to the case 
\beq\label{lambda}0\stackrel{\phi}{\mapsto}\infty\stackrel{\phi}{\mapsto}1\stackrel{\phi}{\mapsto}\lambda_3\stackrel{\phi}{\mapsto}\ldots \stackrel{\phi}{\mapsto}\lambda_{n-1},\eeq
where $\lambda_i\in R_S^*$ is such that there exists $\mu_i\in R_S^*$ so that $\lambda_i+\mu_i=1$.
 By Theorem \ref{csu} we have that the number of the $\lambda_i$'s in the orbit in (\ref{lambda}) such that $\lambda_i\notin k$ is bounded by $\frac{9^{s-1}-1}{2}$. So there are at most $\frac{9^{s-1}+1}{2}$ portions of orbit in (\ref{lambda})of consecutive $\lambda_i$'s where each $\lambda_i\in k$. By Lemma \ref{pairs} each such a portion of orbit in (\ref{lambda}) can not contain more than $2d+1$ elements. Therefore $n\leq \frac{9^{s+1}+1}{2}(2d+1)+2$. 
\end{proof}
To ease notation we shall denote by $A(d,s)$ the number $\frac{9^{s+1}+1}{2}(2d+1)+2$. 
\subsection{Proof of Theorem \ref{Tper}}

We begin with a lemma that is a direct application of Lemma \ref{cp}.

\begin{lemma}\label{prime}
Let $\phi$ be a non isotrivial  rational function defined over $K$ with simple good reduction outside $S$. Let $d$ be the degree of $\phi$. Let $P\in \apl$ be a periodic point for $\phi$ of minimal period $p$ with $p$ a prime number. Then 
$p\leq A(d,s)$, where $A(d,s)$ is the bound given in (\ref{n}).
\end{lemma}
\begin{proof}
Proposition \ref{6.1} affirms that $\dmfp(\phi^i(P),\phi^j(P))=\dmfp(\phi(P),P)$ for each $0\leq j<i<p$. So it is enough to apply Lemma \ref{cp} to the set of points $\{\phi^i(P)\mid 0\leq i\leq p-1\}$.
\end{proof}
Next Lemma bounds the minimal periodicity of the form $p^k$ for a prime $p$. As said in the introduction we use the same ideas applied in the proof of \cite[Theorem 7]{CP}. For the reader's convenience we rewrite those ideas adapted to our situation.

\begin{lemma}\label{powerp}
Let $\phi\in K(z)$ with simple good reduction outside $S$ and  not isotrivial over $K$.  Let $d$ be the degree of $\phi$. Let $P\in \apl$ be a periodic point for $\phi$ of minimal period $p^r$ with $p$ prime number. Then 
$$p^r\leq  \max\left\{\frac{9^{s-1}+1}{2}(2d+1)+2,  p\cdot 3^{2s-1}\right\}.$$
\end{lemma}
\begin{proof}
Up to take a suitable conjugate of $\phi$ by an element in ${\rm PGL}_2(R_S)$ we may assume that $P=0$. To ease notation we denote by $P_i=\phi^i(P)$ for each $0\leq i\leq p^{r}-1$.  By applying Proposition \ref{5.1}, for each $0\leq i\leq p^{r}-1$ and $\mfp\notin S$, we have that $\dmfp(P_i,P_0)\geq\dmfp(P_1,P_0)$, where the equality holds for each $i$ not divisible by $p$. Therefore the cycle of $P_0$ is of the shape 
$$ P_0=0\mapsto x_1/y_1\mapsto A_2x_1/y_2\mapsto \ldots \mapsto A_ix_1/y_i\mapsto\ldots \mapsto x_1/y_{p^r-1}\mapsto [0:1]$$
where each point is written in a $S$--coprime integral form, $A_i\in R_S$ and $A_i=1$ for each $i$ coprime with $p$. If each $A_i$ is in $R_S^*$, by Proposition \ref{6.1} we have that  $\dpid(P_i,P_j)=\dpid(P_0,P_1)$ for each distinct indexes $i,j$ and $\mfp\notin S$. Thus we can apply Lemma \ref{cp} and deduce that $p^k\leq \frac{9^{s-1}+1}{2}(2d+1)+2$. 
Otherwise there exists an integer $j$ such that $A_j\notin R_S^*$. Let $\alpha$ be the smallest integer with this property, so $A_{p^\alpha}$ is not an $S$--unit.

\emph{Case $p=2$.}
For each arbitrary $i\equiv 3\mod 4$ with $0<i<2^r$, we are going to define a solution in $S$--units of the equation 
$\lambda x+y=1$
for a $\lambda\notin R_S^*$ and apply Lemma \ref{goodeq}. 
Let $\alpha>1$. Let $i$ be as above. By Proposition \ref{6.1} we have $\dpid(P_1,P_i)=\dpid(P_0,P_1)=\dpid(P_1,P_{2^\alpha})$, for all $\pid\notin S$. Then
there exist two $S$--units $u_i, u_{2^\alpha}$ such that $P_i=\frac{x_1}{y_1+u_i}$ and $P_{2^\alpha}=\frac{A_{2^\alpha}x_1}{ A_{2^\alpha}y_1+u_{2^\alpha}}$. Again by $\dpid(P_0,P_1)=\dpid(P_i,P_{2^\alpha})$, there exists an $S$--unit $u_{i,\alpha}$ such that $A_{2^\alpha}\frac{u_i}{u_{2^\alpha}}-\frac{u_{i,\alpha}}{u_{2^\alpha}} =1$. By Lemma \ref{goodeq} there are at most $9^{s-1}$ different possible values for $u_i$.
 If $\alpha=1$, for each $i$ as above we have $\dpid(P_1,P_i)=\dpid(P_0,P_2)$ and $\dpid(P_0,P_1)=\dpid(P_1,P_2)$. Then there exist two $S$--units $u_i, u_{2}$ such that $P_i=\frac{x_1}{y_1+A_2u_i}$ and $P_{2}=\frac{A_{2}}{  A_{2}y_1+u_{2}}$. As before, we have $\dpid(P_0,P_1)=\dpid(P_i,P_{2})$.  Hence there exists an $S$--unit $u_{i,2}$ such that $A_{2}^2\frac{u_i}{u_{2}}-\frac{u_{i,2}}{u_{2}} =1$. Again, by Lemma \ref{goodeq}, there exist at most $9^{s-1}$ different possible values for $u_i$.
Note that the positive odd integer $i$ such that $i-1<2^r$ and $4\nmid i-1$ is equal to $2^{r-2}$. Therefore $2^r\leq 4 \cdot 9^{s-1}<3\cdot p\cdot 9^{s-1}$ with $p=2$.

 \emph{Case $p>2$.} Let $b$ be of the form 
 \beq\label{b} b=m\cdot p+i\eeq 
 with $m\in\{0,1,\ldots p^{r-2}\}$ and $i\in\{2,3,\ldots,p-1\}$.   By Proposition \ref{6.1}, we have $\dpid(P_0,P_{b})=\dpid(P_1,P_{b})=v_\pid(x_1)$ , for any $\pid\notin S$. Hence there exists an $S$ unit  $u_{b}$ such that
\beq\label{k}P_{b}=\frac{x_1}{y_1+u_{b}}.\eeq
Since $\dpid(P_1,P_{p^\alpha})=v_\pid(x_1)$, there exists a $S$--unit $u_{p^\alpha}$ verifying $$P_{p^\alpha}=\frac{A_{p^\alpha}x_1}{A_{p^\alpha}y_1+u_{p^\alpha}}.$$
Proposition \ref{6.1} tells us that $\dpid(P_{p^\alpha},P_b)=v_\pid(x_1)$, for every $\pid\notin S$. By identity (\ref{k}), there exists $u_{\alpha, b}\in R_S^*$ such that $A_{p^\alpha} u_b- u_{p^\alpha}=u_{\alpha,b}$. There are exactly $(p^{e-2}+1)(p-2)$ integers $b$ of the form as in (\ref{b}). The pair $(u_b/u_{p^\alpha}, u_{\alpha,b}/u_{p^\alpha})\in (R_S^*)^2$ is a solution of $A_{p^\alpha}x-y=1$, where $A_{p^\alpha}\notin R_S^*$. By Lemma \ref{goodeq}, there are at most only $9^{s-1}$ possible values for $u_b/u_{p^\alpha}$. Hence  $(p^{r-2}+1)(p-2)\leq 9^{s-1}$, i. e. $p^r\leq p^2\left(\frac{9^{s-1}}{p-2}-1\right)<3\cdot p\cdot 9^{s-1}$. 
\end{proof}

We conclude this subsection with the proof of Theorem \ref{Tper}. Let us assume that $\phi\in K(z)$ is not isotrivial over $K$ with good reduction outside $S$. 
By the result due to Bruin and Molnar \cite[Corollary 2.13]{BM} we may assume that $\phi$ is a minimal model of its conjugacy class. Therefore $\phi$ has simple good reduction at a prime $\mfp$ if and only if has good reduction at a prime $\mfp$. Therefore $\phi$ has also simple good reduction outside $S$.

Let us take in consideration the factorization in prime factors of $n$, the minimal period of $P$ as given in the hypothesis of Theorem \ref{Tper}:
\beq\label{factn}n=p_1^{r_1}\cdot \ldots\cdot p_m^{r_m}.\eeq
By applying Lemma \ref{powerp} to the map $\phi^{n/p_i^{r_i}}$, for each prime $p_i$ in the above factorization, we have that the minimal period of $P$ for the map $\phi^{n/p_i^{r_i}}$ is $p_i^{r_i}$. Therefore we have 
\beq\label{bpp}p_i^{r_i}\leq  \max\left\{\frac{9^{s-1}+1}{2}(2d+1)+2,  p_i\cdot 3^{2s-1}\right\}.\eeq
For each prime $p_i$ in the above factorization (\ref{n}), we have that $P$ is a periodic point for $\phi^{n/p_i}$ with minimal period $p_i$. By applying Lemma \ref{prime} to the map $\phi^{n/p_i}$ we have 
\beq\label{bp}p_i\leq  \frac{9^{s-1}+1}{2}(2d+1)+2.\eeq
Putting together (\ref{factn}), (\ref{bpp}) and (\ref{bp}) we have that 
$$n\leq \prod_{p\leq b(d,s)\atop \text{$p$ prime}} \max\left\{\frac{9^{s-1}+1}{2}(2d+1)+2,  p\cdot 3^{2s-1}\right\},$$
where $b(d,s)= \frac{9^{s-1}+1}{2}(2d+1)+2$.

\subsection{Proof of Theorem \ref{Tpreper}} We reduce our problem to the study of the solutions in $S$--units of some equations that will be obtained by applying the divisibility arguments contained in Proposition \ref{5.1}, Proposition \ref{5.2} and Proposition \ref{6.1}. The strategy of the proof is taken from the one in \cite[Theorem 1]{CP}, but here we apply our Lemma \ref{cp}. We shall apply also many times the following result (\cite[Lemma 4.1]{CP}), that is an application of the previous mentioned three propositions. As in the previous subsection we may assume that $\phi$ is a minimal model of its conjugacy class. Therefore assuming that $\phi$ has good reduction outside $S$, it follows that $\phi$ has simple good reduction outside $S$ too. 

\begin{lemma}[Lemma 4.1 in \cite{CP}]\label{pab}
Let
\begin{equation}\label{cn} P=P_{-m+1}\mapsto P_{-m+2}\mapsto \ldots\mapsto P_{-1}\mapsto P_0=[0:1]\mapsto [0:1]\end{equation}
 be an orbit for a rational function $\phi$ defined over $K$, with simple good reduction outside $S$. For any $a,b$ integers such that $0<a<  b\leq m-1$ and $\pid\notin S$,  we have
\beq\label{eqdis}\dpid(P_{-b},P_{-a})=\dpid(P_{-b},P_0)\leq \dpid(P_{-a},P_0).\eeq
\end{lemma}
We may alway assume that $\phi$ and $P=P_0$ are as in the hypothesis of Lemma \ref{pab}. Indeed it is enough to take a suitable iterate $\phi^N$, where $N$ is bounded by the number $C(d,s)$ in (\ref{Bc}) of Theorem \ref{Tper}. By proving a bound $M(d,s)$ for the integer $m$ as in (\ref{cn}) we will prove that we may take as $D(d,s)$ any number such that 
\beq\label{Dds}D(d,s)\geq C(d,s)\cdot (M(d,s)+2).\eeq
So the proof of Theorem \ref{Tpreper} will follows from the following result.

\begin{lemma}\label{Bds}
Let $\phi$ and the $P_i$'s be as in Lemma \ref{pab}. Assume that $\phi$ is not isotrivial over $K$. Then 
$$m\leq A(d,s)+9^{s-1}-1,$$ 
where the number $A(d,s)$ is the one given in (\ref{n}) of Lemma \ref{cp}.
\end{lemma}
\begin{proof}
For each index $-m\leq i\leq 1$ we assume that $P_i=x_i/y_i$ is written in $S$--coprime integral form. By Lemma \ref{pab} we have that for each $-m\leq -j<-i\leq $ there exists an $S$--integer $T_{i,j}$ such that $x_i=T_{i,j}x_j$. 

Consider the $\pid$--adic distance between the points $P_{-1}$ and $P_{-j}$. Again by Lemma \ref{pab}, we have
$$\dpid(P_{-1},P_{-j})=v_\pid(x_1y_j-x_1y_1/T_{1,j})=v_\pid(x_1/T_{1,j}),$$
for all $\pid\notin S$. Then, there exists an $S$--unit $u_j$ such that
$y_j=\left(y_1+u_j\right)/{T_{1,j}}$. Thus
\beq\label{Pj} P_{-j}=\frac{x_i/T_{1,j}}{\left(y_1+u_j\right)/{T_{1,j}}}=\frac{x_1}{y_1+u_j}.\eeq

Note that the maximal index $N$ such that
\beq\label{N}\dmfp(P_{-N},P_0)=\dmfp(P_{-1},P_0)\eeq
for all $\mfp\notin S$ is such that 
$N+1\leq A(d,s)$, where $A(d,s)$ is the number in (\ref{n}). Indeed by applying Lemma \ref{pab} we have that $\dmfp(P_{-i},P_{-j})=\dmfp(P_{-1},P_0)$ for each indexes $-N\leq -i< -j\leq 0$. Therefore if $m=N$ we are done. Note that condition (\ref{N}) with $N=m$ implies that $T_{1,i}\in R_S^*$ for each $-m\leq -i\leq -1$.  

Suppose that there exists a index $-m\leq -a< -1$ such that $T_{1,a}\notin R_S^*$. Let $a$ be the minimum index with this property. 
Consider the $\pid$--adic distance between the points $P_{-a}$ and $P_{-b}$ for each $a<b$.  By Lemma \ref{pab} and by (\ref{Pj}), we have:
$\dpid(P_{-b},P_{-a})=v_\pid\left(x_a((y_1+u_b)/T_{1,b})-(x_1/T_{1,b})y_a\right)=v_\pid(x_1/T_{1,b})=\dmfp(P_{-b},P_0)$,
for all $\pid\notin S$. Then there exists $w_b\in R_S^*$ such that
\beq\label{fe}\frac{x_1}{x_a{y_1}-x_1{y_a}}w_b-\frac{x_a}{x_a{y_1}-x_1{y_a}}u_b=1.\eeq
Note that this last equation satisfies the hypothesis of Lemma \ref{goodeq}, so there are at most $9^{s-1}$ possibilities for the $S$--unit $u_b$. Therefore by (\ref{Pj}) the point $P_{-b}$ may assume at most $9^{s-1}$ possibilities. The previous arguments about the integer $N$ tells us that $a+1\leq A(d,s)$. Therefore we have that 
$m\leq a+9^{s-1}\leq A(d,s)+9^{s-1}-1$.
\end{proof}
From Lemma \ref{Bds} we see that the bound $M(d,s)$ in (\ref{Dds}) is equal to $A(d,s)+9^{s-1}-1$. Therefore we can take
$$D(d,s)=\frac{9^{s+1}+1}{2}(d^2+2)\prod_{p\leq p_{\pi(b(d,s))}\atop \text{$p$ prime}}\max\left\{\frac{9^{s-1}+1}{2}(2d+1)+2,  p\cdot 3^{2s-1}\right\}$$
that is bigger than $(A(d,s)+9^{s-1}+1)\cdot C(d,s)$.

\subsection{Proof of Theorem \ref{mainT}}
It is possible to bound the cardinality of the set of periodic points in $\p_1(K)$ in terms only of $d$ and $s$. Let $C(d,s)$ and $D(d,s)$ the bounds given in Theorem \ref{Tper} and Theorem \ref{Tpreper} respectively. Let 
$$N(d,s)=\prod_{p \ \text{prime}}p^{m_p(C(d,s))},$$
where $m_p(C(d,s))=\max \{{\rm ord}_p(z)\mid z\in\NN,\ z\leq C(d,s)\}$. Therefore, by Theorem \ref{Tper},  the cardinality of the set of periodic points in $p_1(K)$ for $\phi$ is bounded by $d^{N(d,s)}+1$. By Theorem \ref{Tpreper}, the set ${\rm PrePer}(\phi,K)$ has cardinality bounded by $d^{D(d,s)}\left(d^{N(d,s)}+1\right)$.

\end{document}